\documentclass[11pt]{article}

\usepackage[usenames]{color}
\usepackage{amssymb,amsmath,latexsym}
\usepackage{amsthm}
\usepackage[mathscr]{eucal}
\usepackage{setspace}
\usepackage[color=yellow,textsize=small]{todonotes}
\usepackage{geometry}
 \usepackage{multirow}
 \usepackage{graphicx}
 \usepackage{float}
\newtheorem{lemma}{Lemma}
\newtheorem{theorem}{Theorem}

\newtheorem{proposition}{Proposition}

\newtheorem{thm}{Theorem}[section]

\newtheorem{lem}[thm]{Lemma}

\newtheorem{rem}[thm]{Remark}

\newcounter{bean}
\newcommand{\benuma}{\setlength{\labelwidth}{.25in}
\begin{list}%
{(\alph{bean})}{\usecounter{bean}}}
\newcommand{\eenuma}{\end{list}}
\newcommand{\beginsec}{{\setcounter{equation}{0}}{\setcounter{theorem}{0}}
{\setcounter{lemma}{0}}{\setcounter{definition}{0}}{\setcounter{remark}{0}}
{\setcounter{proposition}{0}} {\setcounter{corollary}{0}} }

\newcommand{\noi}{\noindent}
\newcommand{\bi}{\begin{itemize}}
\newcommand{\ei}{\end{itemize}}
\newcommand{\be}{\begin{enumerate}}
\newcommand{\ee}{\end{enumerate}}
\newcommand{\beqs}{\begin{equation*}}
\newcommand{\eeqs}{\end{equation*}}
\newcommand{\beq}{\begin{equation}}
\newcommand{\eeq}{\end{equation}}
\newcommand{\beqys}{\begin{eqnarray*}}
\newcommand{\eeqys}{\end{eqnarray*}}
\newcommand{\beqy}{\begin{eqnarray}}
\newcommand{\eeqy}{\end{eqnarray}}

\newcommand{\E}{\mathbb E}

\newcommand{\pos}{\in(0,\infty)}

\newcommand{\ra}{\rightarrow}
\newcommand{\Ra}{\Rightarrow}

\newcommand{\sig}[2]{\sum_{#1}^{#2}}
\newcommand{\inte}[2]{\int_{#1}^{#2}}
\newcommand{\atob}[2]{{#1}^{#2}}

\newcommand{\si}{\sigma}

\newcommand{\wh}{\widehat}

\newcommand{\var}{\operatorname{Var}}

\mathchardef\mhyphen="2D
\newcommand{\bp}{\begin{pmatrix}}
\newcommand{\ep}{\end{pmatrix}}
\begin{document}

\title{On drift parameter estimation for reflected fractional Ornstein-Uhlenbeck processes}

\author{Chihoon Lee\footnote{E-mail: chihoon@stat.colostate.edu}\\
Department of Statistics\\
Colorado State University\\
Fort Collins, CO 80523, USA \\ \and Jian Song\footnote{E-mail:  txjsong@hku.hk} \\
Department of Mathematics\\
University of Hong Kong\\ Hong Kong}

\date{}
\maketitle


\begin{abstract}
{\color{black} \noi We consider a reflected Ornstein-Uhlenbeck process $X$ driven by a fractional Brownian motion with Hurst parameter $H\in (0, \frac12) \cup (\frac12, 1)$.  Our goal is to estimate an unknown drift parameter $\alpha\in (-\infty,\infty)$ on the basis of continuous observation of the state process.  We establish Girsanov theorem for the process $X$, derive the standard maximum likelihood estimator of the drift parameter $\alpha$, and prove its strong consistency and asymptotic normality.  As an improved estimator, we obtain the explicit formulas for the sequential  maximum likelihood estimator and its mean squared error by assuming the process is observed until a certain information reaches a specified precision level.  The estimator is shown to be unbiased, uniformly normally distributed, and efficient in the mean square error sense.}
\end{abstract}
\noi {\it Keywords: Reflected fractional Ornstein-Uhlenbeck processes, fractional Brownian motion, fractional calculus, parameter estimation, maximum likelihood estimator, sequential maximum likelihood estimator.}

\noi {\it   AMS Subject Classifications:} Primary 60G22; secondary 60H30, 62M09, 90B22

\beginsec
\section{Introduction}
We consider a drift parameter estimation problem for  a one-dimensional reflected fractional Ornstein-Uhlenbeck (RFOU) process with infinitesimal drift $-\alpha x$ and infinitesimal variance $\atob{\sigma}{2}$, where $\alpha\in (-\infty,\infty)$ and $ \sigma>0$. The RFOU process can serve as approximating models in diverse applications such as in physical, biological, and mathematical finance models (see, e.g., \cite{Ricciardibook,Attia91,BoTang2010,BoWangYang2011} and also Section \ref{moti} below). The RFOU behaves as a standard FOU process in the interior of its domain \textcolor{black}{$(0,\infty)$}. However, when it reaches its boundary \textcolor{black}{at zero}, then the sample path returns to the interior in a manner exercising with minimal ``pushing'' force. Our main interest in this model stems from the fact that the RFOU process arises as the key approximating process for queueing systems with reneging or balking customers with long range dependent inter-arrival and/or service time processes  (see \cite{WardGlynn2003, {KonstanLin1996StoNetworks}, WardGlynn2005} and the references therein).  In such cases, the drift parameter $\alpha$ carries the physical meaning of customers' reneging (or, balking) rate from the system.  More details are provided in Section \ref{moti} with regards to how the RFOU model can arise in the applications.  In this paper, we expand the previously known results on the parameter estimation problems for a reflected Ornstein-Uhlenbeck (ROU) process to the case when the noise process is given by a fractional Brownian motion with Hurst parameter $H\in (0, \frac12) \cup (\frac12, 1)$. Such results require several nontrivial technical efforts, since the fractional Brownian motion is not a Markov process nor a semimartingale (unless $H=1/2$) and the classical stochastic calculus is inapplicable in its analysis.  

We describe the RFOU model more precisely. Let {\color{black}$\Lambda:=(\Omega, \mathcal F, (\mathcal F_t)_{t\geq0}, P)$} be a complete filtered probability space with the filtration $(\mathcal F_t)_{t\geq0}$ satisfying the \emph{usual} conditions.  Define the RFOU process $\{X_t:t\geq0\}$ reflected at \textcolor{black}{zero} on $\Lambda$ as follows.  {\color{black} Let $\{X_t:t\geq0\}$ be the strong solution  to the stochastic differential equation:}  \beq\label{0.1} \left. \begin{gathered} dX_t=-\alpha X_t dt + \sigma dW^H_t + dL_t, \\ X_t\geq \textcolor{black}{0} \quad \mbox{for all}\quad t\geq0,  \\ X_0=x, \end{gathered}  \\ \right\} \eeq where $\alpha\in(-\infty,\infty)$, $\sigma\pos$, \textcolor{black}{$x\in[0,\infty)$} and $W^H=(W^H_t)_{t\geq0}$ is a one-dimensional standard fractional Brownian motion on $\Lambda$ with {a known} Hurst index $H\in (0, \frac12) \cup (\frac12, 1)$. Here, the process $L=(L_t)_{t\geq0}$ is defined to be the minimal, non-decreasing and non-negative process with initial value $L_0=0$, which increases only when $X$ hits the boundary \textcolor{black}{$0$}, so that \beq\label{0.2} \inte{[0,\infty)}{} I(\textcolor{black}{X_t>0})dL_t=0, \eeq where $I(\cdot)$ is the indicator function. \textcolor{black}{That is,  $L_t$ represents the (cumulative) local time of $X$ at the boundary zero up to time $t\geq0$.}  

\textcolor{black}{The existence and uniqueness of a pathwise solution to \eqref{0.1} follows from the standard Picard iteration and uniqueness of the one-dimensional reflection (Skorohod) map. Statistical properties (such as being Markov or semimartingale) of underlying noise processes are immaterial to such results. More generally, as long as the drift function is Lipschitz and the sample paths of noise process are right continuous with left limits, the existence and uniqueness of a pathwise solution are guaranteed (see, for instance, 
\cite{MR3076769, Harrison-book}).}


Our main interest lies in the statistical inference for the RFOU process \eqref{0.1}. More specifically, our aim is to estimate the unknown drift parameter $\alpha\in(-\infty,\infty)$ in \eqref{0.1} based on observation of the state process $\{X_t\}_{t\geq0}$.  We assume that the infinitesimal variance parameter $\sigma^2$ is known as this can be estimated via a standard approach involving the associated quadratic variation process; see Remark \ref{infivar}.

 In \cite{MR1943832}, by applying the approach of fundamental martingale for fractional Brownian motion (see \cite{NVV99}) and computing the Laplace transform of a functional of the state process, the authors investigated an asymptotic behavior of the maximum likelihood estimator (MLE) of the drift parameter for a standard (non-reflected) FOU process when $H \in [1/2, 1)$. They obtained the strong consistency of the estimator and explicit formulas for the asymptotic bias and mean square error. In \cite{MR2644034}, the authors established long time asymptotic properties (such as consistency, asymptotic normality and convergence of the moments) of the MLE for the signal drift parameter in a partially observed fractional diffusion system, via the computations of Laplace transform.  Sharp large deviation properties of the energy and the MLE for the FOU process with $H\in(1/2,1)$ were studied in \cite{MR2859161}. For parameters in stochastic partial differential equations (SPDEs) with additive space-time white noise,  properties of MLE were investigated in \cite{MR1427311, MR1355054}.  When the driving noise process of SPDEs is white in space and fractional in time, the properties of MLE were studied in \cite{MR2531625}. More details on the statistical inference for SPDEs can be found in the survey paper \cite{MR2474113} and the references therein.
 
 The authors of \cite{jspi} studied the MLE for the model (\ref{0.1}) when $\alpha\in(0,\infty)$ (i.e., the ergodic case) and $H=1/2$ (i.e., the standard Brownian motion case), and established several important properties.  The MLE $\wh \alpha_T$  of $\alpha$, based on the process $\{X_t\}$ up to a previously determined fixed time $T$, is given by \beq\label{0.5} \widehat \alpha_T:=\textcolor{black}{\frac{-\inte{0}{T}X_tdX_t}{\inte{0}{T}\atob{X}{2}_tdt}}\,. \eeq  The MLE $\wh \alpha_T$ satisfies strong consistency and asymptotic normality as $T\ra\infty$.  However, this estimator is biased and its mean squared error (MSE) depends on the unknown parameter to be estimated.  We note that exact estimates for the bias and the MSE of the estimator $\wh \alpha_T$ are not available.  As a remedy for this, a sequential estimation plan $(\tau(h),\wh \alpha_{\tau(h)})$ was proposed in \cite{LeeBL2012}. It is assumed in \cite{LeeBL2012} that the parameter ranges the whole real line $\alpha\in(-\infty,\infty)$ (i.e., it covers ergodic, non-ergodic, non-stationary cases) and the process $\{X_t\}$ is observed until the observed Fisher information of the process exceeds a predetermined level of precision $h$ (see also \cite{BoYang-spl2012}). More precisely, $\{X_t\}$ is observed over the random time interval $[0,\tau(h)]$ where the stopping  time $\tau(h)$ is defined as \beq\label{1.1} \tau(h):=\inf\left\{t\geq0: \inte{0}{t} \atob{X}{2}_sds\geq h \right\}, \quad 0<h<\infty, \eeq and the  $\atob{\mathcal F}{X}_{\tau(h)}$-measurable function $\wh \alpha_{\tau(h)}$ defined by \beq\label{1.2}\wh \alpha_{\tau(h)}:=\textcolor{black}{-\frac{1}{h}\inte{0}{\tau(h)}X_sdX_s} \eeq  is a sequential estimator. Then the sequential estimation plan $(\tau(h),\wh \alpha_{\tau(h)})$ has shown to satisfy the following properties (cf. Chapter 17.5 of \cite{LipShir2001v2} or Chapter 5.2 of \cite{Bishwalbook}): (a) it is unbiased; (b) the plan is closed, i.e., the time of the observation $\tau(h)$ is finite with probability 1; (c) its MSE is a constant that does not depend on the parameter to be estimated; (d) not only it provides consistent estimation plan but also $\wh\alpha_{\tau(h)}$ is exactly normally distributed, which makes it possible to construct an exact confidence interval for the parameter $\alpha$.

Our main results are concerned with extending the aforementioned estimators \eqref{0.5}, \eqref{1.2} and their respective statistical properties to the case when the system is driven by a fractional Brownian motion with \textcolor{black}{$H\in (0, \frac12) \cup (\frac12, 1)$}.    We establish uniform exponential moment estimates of the RFOU process, which, in conjunction with certain integral representations and the fundamental martingales of fractional Brownian motions, leads to two types of fractional Girsanov formulas. Then, we obtain the standard MLE and prove its strong consistency and asymptotic normality. Furthermore, we derive the explicit expression for the sequential  MLE and show that it is unbiased, uniformly normally distributed (over the entire parameter space which is the real line), and efficient in the mean square error sense.

The rest of the paper is organized as follows. In Section \ref{moti}, we provide a brief discussion and derivation about how the RFOU model can naturally arise in the applications. Section \ref{prim} is devoted to preliminaries on fractional Brownian motion and fractional calculus that are necessary in our analysis.  In Section \ref{paraest}, we establish a Girsanov theorem for the RFOU process in Theorem \ref{Girsanov2} and obtain the standard maximum likelihood estimator of the drift parameter, and then prove its strong consistency and asymptotic normality. We also develop an equivalent version of the Girsanov theorem in Theorem \ref{Girsanov1}, directly from Theorem \ref{Girsanov2} by using the transformation result {\color{black} introduced in Section \ref{fbm}}. In Section \ref{seq}, we derive the explicit formulas for the sequential  maximum likelihood estimator and its mean squared error by assuming the process is observed until a certain information reaches a specified precision level.  The sequential estimator is shown to be unbiased, uniformly normally distributed, and efficient in the mean square error sense.  Finally, in the Appendix we provide an independent, more direct proof for Theorem \ref{Girsanov1} and an auxiliary result for fractional Brownian motion.

\setcounter{equation}{0}
\section{Motivation of the RFOU model}\label{moti}
Here we provide some details about how the RFOU model could arise in the applications. Firstly, in the context of financial time series modelling, the RFOU processes can be used to describe the spot foreign exchange rate processes, the domestic interest rate processes, and even some asset price processes in a regulated financial market system (cf. \cite{BoWangYang2011, BoTang2010}) with long range dependence and heavy tails stylized facts, which seem to be common to a wide variety of markets, instruments and periods \cite{Cont_longrange, WillingerTaqquTever}. Secondly, in engineering applications to queueing and storage systems, the RFOU model can play as the key approximating process for systems with reneging or balking customers (cf. \cite{WardGlynn2003, WardGlynn-ROU2003, WardGlynn2005} and the references therein), subject to their inter-arrival and/or service time processes exhibiting long range dependence characteristics in the traffic data. In such cases, the drift parameter  $\alpha$ carries the physical meaning of customers' reneging (or, balking) rate from the system. We provide a more detailed justification below.

Consider a single-server, single-class queueing model under heavy traffic subject to their reneging behaviors.  More precisely, we shall consider a sequence of single-server queueing systems indexed by $N=1, 2, \ldots$, and assume that the sequence of interarrival times $\{T^{(N)}_k - T^{(N)}_{k-1}\}_{k=1}^\infty$ are stationary with $E(T^{(N)}_k-T^{(N)}_{k-1})=1/\lambda^{(N)}$ such that $\lambda^{(N)}\ra\lambda\pos$ as $N\ra\infty$. Also, we assume an invariance principle holds: \[ \frac{T^{(N)}_{[Nt]} - [Nt]/\lambda^{(N)}}{N^H} \Ra \sigma W^H(t), \] where $W^H$ is a standard fractional Brownian motion with Hurst parameter $H$ and $EW^H(t)^2= t^{2H}$ and some scaling constant $\sigma>0$. For the sake of simplicity, we will assume $\lambda=1$.    Then, under a mild assumption on $\var(T^{(N)}_k-T^{(N)}_{k-1})\leq K<\infty$, one gets the following functional central limit theorem for the arrival process $A^{(N)}(t)=\sig{n=1}{\infty}I(T^{(N)}_n\leq t)$, $t\geq0$: \beq\label{fc} \frac{A^{(N)}(Nt) - \lambda^{(N)}Nt}{N^H} \Ra \sigma W^H(t), \eeq where $W^H$ is a standard fractional Brownian motion (cf. Theorem 2.1 of \cite{KonstanLin1996StoNetworks}).

With the arrival process $\{A^{(N)}(t): t\geq0\}$, consider the queueing system operating with a constant service rate $\mu^{(N)}>0$, if there are customers in the buffer, and otherwise the server becomes idle. Moreover, customers faced with long waiting times will abandon the system before receiving service; let $R^{(N)}(t)$ be the total number of customers who abandoned the system during the time interval $[0, t]$. Incorporating these conditions will yield the following equation on the queue length process $(Q^{(N)}(t): t\geq0)$: \beqys Q^{(N)}(t) &=& Q^{(N)}(0)  + A^{(N)}(t) - R^{(N)}(t) - \inte{0}{t} \mu^{(N)} I_{[Q^{(N)}(s)>0]}ds \\ &=& Q^{(N)}(0)  + A^{(N)}(t) - \inte{0}{t} \alpha^{(N)}Q^{(N)}(s)ds - \mu^{(N)}t + \inte{0}{t} \mu^{(N)} I_{[Q^{(N)}(s)=0]}ds,  \eeqys where $\alpha^{(N)}>0$ is a proportionality factor related with the customers' abandonment rate from the system.  We assume that the abandonment factor $\alpha^{(N)}$ is of $O(N^H)$, in particular, $\alpha^{(N)}/N^H\ra\alpha$ as $N\ra\infty$ for some constant $\alpha>0$.

In view of the functional central limit theorem scaling used in \eqref{fc}, we have that \beqys \frac{Q^{(N)}(Nt)}{N^H} &=& \frac{Q^{(N)}(0)}{N^H}  + \frac{A^{(N)}(Nt)}{N^H} - \frac{1}{N^H}\inte{0}{Nt} \alpha^{(N)}Q^{(N)}(s)ds  \\ && -  \frac{\mu^{(N)}Nt}{N^H} + \frac{1}{N^H}\inte{0}{Nt} \mu^{(N)} I_{[Q^{(N)}(s)=0]}ds \\ &=& \frac{Q^{(N)}(0)}{N^H}  + \frac{A^{(N)}(Nt)- \lambda^{(N)}Nt}{N^H} - \frac{1}{N^H}\inte{0}{Nt} \alpha^{(N)}Q^{(N)}(s)ds  \\ && -  \frac{\mu^{(N)}Nt- \lambda^{(N)}Nt}{N^H} + \frac{1}{N^H}\inte{0}{Nt} \mu^{(N)} I_{[Q^{(N)}(s)=0]}ds.  \eeqys Lastly, we impose a ``heavy traffic'' assumption implying that the system processing capacity is balanced with the system load, that is,  the ``drift'' term $(\lambda^{(N)}-\mu^{(N)})$ converges to zero, at a certain rate closely related with the scaling in \eqref{fc}:  $N^{1-H}(\lambda^{(N)}-\mu^{(N)})\ra 0$ as $N\ra\infty$. With the initial condition $\frac{Q^{(N)}(0)}{N^H} \ra x\in [0,\infty)$ and owing to the Lipschitz continuity property of the \emph{generalized} Skorohod (reflection) map \cite{WardGlynn2005, ReedWard2004} together with the continuous- mapping theorem, we finally obtain the weak convergence of the scaled queue length process $\{\frac{Q^{(N)}(Nt)}{N^H}: t\geq0\}_{N=0}^\infty$ to the RFOU process with the reflecting boundary given by zero.  We note that similar derivations are possible for the offered waiting time process, in the context of state-dependent admission control setup with customers' impatient behaviors (see \cite{LeeWeerasinghe2011SPA} and the references therein).

\setcounter{equation}{0}

\section{Preliminaries on  fractional calculus and fractional Brownian motion}\label{prim}
\subsection{Fractional calculus}
In this section, we recall some basic results from  fractional calculus. See \cite{MR1347689} for more details.
Let $a,b\in \mathbb{R}$ with $a<b$ and let $ \alpha >0.$ (The symbol $\alpha$ in this section should not be confused with the parameter of the RFOU process.)  The
left-sided  and right-sided  fractional Riemann-Liouville integrals
of $f\in L^1([a,b])$ of order $\alpha$   are   defined for almost all $t\in(a,b)$ by
\[
I_{a+}^{\alpha }f\left( t\right) =\frac{1}{\Gamma \left( \alpha \right) }%
\int_{a}^{t}\left( t-s\right) ^{\alpha -1}f\left( s\right)
ds\,,
\]%
and
\[
I_{b-}^{\alpha }f\left( t\right) =\frac{\left( -1\right) ^{-\alpha
}}{\Gamma \left( \alpha \right) }\int_{t}^{b}\left( s-t\right)
^{\alpha -1}f\left( s\right) ds,
\]%
respectively, where $\left( -1\right) ^{-\alpha }=e^{-i\pi \alpha }$ and $%
\displaystyle\Gamma \left( \alpha \right) =\int_{0}^{\infty }r^{\alpha
-1}e^{-r}dr$ is the Euler gamma function.

Let $I_{a+}^\alpha(L^p)$ (resp. $I_{b-}^\alpha(L^p)$) be the image of $L^p([a,b])$ by the operator $I_{a+}^\alpha$ (resp. $I_{a+}^\alpha$).

If $f\in I_{a+}^\alpha(L^p)$ (resp. $f\in I_{b-}^\alpha(L^p)$ ) and $\alpha\in(0,1)$, then the left and right-sided  fractional derivatives are defined as
\begin{equation}
D_{a+}^{\alpha }f\left( t\right) =\frac{1}{\Gamma \left( 1-\alpha \right) }%
\left( \frac{f\left( t\right) }{\left( t-a\right) ^{\alpha }}+\alpha
\int_{a}^{t}\frac{f\left( t\right) -f\left( s\right) }{\left(
t-s\right) ^{\alpha +1}}ds\right)  \label{e.2.1}
\end{equation}%
and
\begin{equation}
D_{b-}^{\alpha }f\left( t\right) =\frac{\left( -1\right) ^{\alpha
}}{\Gamma \left( 1-\alpha \right) }\left( \frac{f\left( t\right)
}{\left( b-t\right)
^{\alpha }}+\alpha \int_{t}^{b}\frac{f\left( t\right) -f\left( s\right) }{%
\left( s-t\right) ^{\alpha +1}}ds\right)  \label{e.2.2}
\end{equation}%
for almost all $t\in(a,b)$.

Let $C^\alpha([a,b])$ denote the space of $\alpha$-H\"older continuous functions of order $\alpha$ on the interval $[a,b]$. When $\alpha p>1$, then we have $I_{a+}^\alpha(L^p)\in C^{\alpha-\frac1p}([a,b])$. On the other hand, if $\beta>\alpha$, then $C^\beta([a,b])\subset I_{a+}^\alpha(L^p)$ for all $p>1$.

The following inversion formulas hold:
\begin{align*}
&I_{a+}^\alpha(I_{a+}^\beta f)=I_{a+}^{\alpha+\beta} f, & f\in L^1;\\
&D_{a+}^\alpha(I_{a+}^\alpha f)=f, & f\in L^1;\\
&I_{a+}^\alpha(D_{a+}^\alpha f)=f, & f\in I_{a+}^\alpha(L^1);\\
&D_{a+}^\alpha(D_{a+}^\beta f)=D_{a+}^{\alpha+\beta} f, & f\in I_{a+}^{\alpha+\beta}(L^1) , \alpha+\beta\le1.\\
\end{align*}
Similar inversion formulas hold for the operators $I_{b-}^\alpha$ and $D_{b-}^\alpha$ as well.

We also have the following integration by parts formula.
\begin{proposition}
If $f\in I_{a+}^\alpha(L^p),g\in I_{b-}^\alpha(L^q)$ and $\frac1p+\frac1q=1$, we have
\begin{equation}\label{ibp}
\int_a^b (D_{a+}^\alpha f)(s)g(s)ds=\int_a^b f(s)(D_{b-}^\alpha g)(s)ds.
\end{equation}
\end{proposition}
The following proposition indicates the relationship between Young's integral and Lebesgue integral.
\begin{proposition}
\label{p1} Suppose that $f\in C^{\lambda }(a,b)$ and $g\in C^{\mu
}(a,b)$ with $\lambda +\mu >1$. Let ${\lambda }>\alpha $ and $\mu
>1-\alpha $. Then the Riemann-Stieltjes integral $\int_{a}^{b}fdg$
exists and it can be
expressed as%
\begin{equation}
\int_{a}^{b}fdg=(-1)^{\alpha }\int_{a}^{b}D_{a+}^{\alpha }f\left(
t\right) D_{b-}^{1-\alpha }g_{b-}\left( t\right) dt\,,
\label{e.2.4}
\end{equation}%
where $g_{b-}\left( t\right) =g\left( t\right) -g\left( b\right) $.
\end{proposition}

\subsection{Fractional Brownian motion}\label{fbm}

{\color{black} Let $H\in (0, 1)$ be a constant. A fractional Brownian motion $\{W_t^H, t\ge 0\}$  of Hurst parameter $H$ is a  Gaussian process with zero mean and covariance function}
\[\E(W_t^HW_s^H)=\frac12(t^{2H}+s^{2H}-|t-s|^{2H})=:R_H(t,s).\]
Denote by $\mathcal E$ the set of step functions on $[0,T]$. Let $\mathcal H$ be the Hilbert space defined as the closure of $\mathcal E$ with respect to the scalar product
$\langle I_{[0,t]},I_{[0,s]}\rangle=R(t,s).$ Then the mapping $I_{[0,t]}\to W_t^H$ can be extended to be an isometry $\varphi\to W^H(\varphi)$ between $\mathcal H$ and the Gaussian space generated by $W^H$.

The covariance kernel $R_H(t,s)$ has the following integral representation
\[R_H(t,s)=\int_0^{t\wedge s} K_H(t,r)K_H(s,r)dr,\]
where
\begin{equation}\label{e.K}
K_H(t,s)=b_H\left[(\frac ts)^{H-\frac12}(t-s)^{H-\frac12}-(H-\frac12)s^{\frac12-H}\int_s^t u^{H-\frac32}(u-s)^{H-\frac12}du\right],
\end{equation}
with $b_H=\left(\dfrac{2H\Gamma(\frac32-H)}{\Gamma(H+\frac12)\Gamma(2-2H)}\right)^\frac12,$ where $\displaystyle\Gamma(x)=\int_0^\infty s^{x-1}e^{-s}ds$ for $ x>0$ is the Gamma function.

In particular, when $H\in(\frac12,1), K_H(t,s)$ can be all written as
\[K_H(t,s)=C_Hs^{\frac12-H}\int_s^t(u-s)^{H-\frac32}u^{H-\frac12}du,\]
with $C_H=\left(\dfrac{H(2H-1)}{\beta(2-2H,H-\frac12)}\right)^{\frac12}=b_H(H-\frac12),$ where $\displaystyle\beta(x,y)=\int_0^1t^{1-x}(1-t)^{1-y}dt$ for $ x>0, y>0$ is the Beta function.

Let $K_H$ denote the operator  on  $L^2([0,T])$, which is an isomorphism from $L^2([0,T])$ onto $I_{0+}^{H+\frac12}(L^2([0,T]))$,
\[(K_Hf)(t)=\int_0^t K_H(t,s)f(s)ds.\]
We can rewrite the action of $K_H$ as (see \cite{MR1677455} or \cite{MR1934157})
\begin{equation}\label{e.Koperator}
(K_Hf)(t)=
\begin{cases} C_H\Gamma(H-\frac12)I_{0+}^{1}s^{H-\frac12}I_{0+}^{H-\frac12}(s^{\frac12-H}f)(t),& \text{ if }\,\, H> \frac12\,,\\
b_H\Gamma(H+\frac12)I_{0+}^{2H}s^{\frac12-H}I_{0+}^{\frac12-H}(s^{H-\frac12}f)(t), &\text{ if }\,\, H< \frac12\,.
\end{cases}
\end{equation}
If $f$ is absolutely continuous, for $H<\frac12$, we can write
\[(K_Hf)(t)=b_H\Gamma(H+\frac12)I_{0+}^{1}s^{H-\frac12}D_{0+}^{H-\frac12}(s^{\frac12-H}f)(t).\]
Consider the operator $K_H^*$ from $\mathcal E$ to $L^2([0,T])$ defined as
\[(K_H^*\varphi)(s)=K_H(T,s)\varphi(s)+\int_s^T(\varphi(r)-\varphi(s))\dfrac{\partial K_H}{\partial r}(r,s)dr.\]
Noting that $(K_H^*I_{[0,t]})(s)=K_H(t,s)I_{[0,t]}(s), K_H^*$ can be extended to be an isometry between $\mathcal H$ and $L^2([0,T])$,
\[\langle f, g \rangle_\mathcal H=\langle K_H^*f, K_H^*g\rangle_{L^2([0,T])}, \quad f,g\in \mathcal H.\]

As a consequence, the operator $K_H^*$ provides an isometry between the Hilbert spaces $\mathcal H$ and $L^2([0,T])$. Hence the process
\begin{equation}\label{bm}
W_t=W^H((K_H^*)^{-1}I_{[0,t]}), \quad t\ge0
\end{equation} is a standard Brownian motion, and $W^H$ has the following  integral representation
{\color{black} \begin{equation}\label{eq3.4'}
W_t^H=\int_0^t K_H(t,s)dW_s.
\end{equation}
We say that $W$ is the standard Brownian motion \emph{related with} $W^H$}.

Moreover, we have the relationship between the Wiener integrals with respect to fractional Brownian motion and its related Brownian motion as
\begin{align}\label{eq3.4}
W^H(\varphi):=\int_0^T \varphi(s)dW_s^H=\int_0^T (K_H^*\varphi)(t)dW_t,\quad\varphi\in \mathcal H.
\end{align}
{\color{black}  We refer to \cite[Chapter 2]{MR2387368} and \cite[Section 5.1]{MR2200233} for more details.}

{\color{black} The expressions in (\ref{bm}) and (\ref{eq3.4'}) provide the transformations between fractional Brownian motion and Brownian motion. In this section, we shall establish another useful transformation using the fundamental martingale $M^H$ defined in (\ref{e4.0}) below.}

{\color{black} For $H\in (0, \frac12) \cup (\frac12, 1)$, define
\begin{align*}
k_H(t,s)=\kappa_H^{-1}s^{\frac12-H}(t-s)^{\frac12-H},\, 0<s<t,
\end{align*}
where $\kappa_H=2H\Gamma(\frac32-H)\Gamma(H+\frac12).$ Define
\begin{equation}\label{e4.0}
M_t^H=\int_0^t k_H(t,s)dW_s^H.
\end{equation}}

 Then from (\ref{eq3.4}), one gets that  $M^H$ is a Gaussian martingale with  quadratic variation $\langle M^H \rangle_t=\lambda_H^{-1}t^{2-2H},$ where $\lambda_H=\frac{2H\Gamma(3-2H)\Gamma(H+\frac12)}{\Gamma(\frac32-H)}.$ The martingale $M^H$ was introduced as the \emph{fundamental martingale} for the fractional Brownian motion $W^H$ in order to get a Girsanov type theorem for $W^H$ in \cite{NVV99} .

By L\'evy Characterization Theorem, it is easy to verify that the process
\begin{equation}\label{ee4.1}
B_t=\frac{H(2H-1)}{C_H} \int_0^t s^{H-\frac12}dM_s^H, \quad t\geq0
\end{equation}
is a standard Brownian motion {\color{black} for $H\in (0, \frac12) \cup (\frac12, 1)$}.

The formula (\ref{eq3.4'}) provides an integral representation for the fractional Brownian motion $W^H$ in terms of the standard Brownian motion $W$ given in (\ref{bm}). On the other hand, Theorem \ref{connection} below adopted from \cite[Theorem 5.2]{NVV99} claims that the Brownian motion $W$ coincides with $B$ given in (\ref{ee4.1}) pathwisely.  
\begin{thm}\label{connection}
Let the process $B=\{B_t\}_{t\geq0}$ be defined as in \eqref{ee4.1}. Then $B$ is the standard Brownian motion \emph{related} with $W^H$ {\color{black} for $H\in (0, \frac12) \cup (\frac12, 1)$}, i.e.,
\[W_t^H=\int_0^t K_H(t,s)dB_s,\]
{\color{black} and hence $B_t= W_t$ for $t\ge 0$ a.s. where $W$ is given by (\ref{bm})}.
\end{thm}

\setcounter{equation}{0}
\section{Parameter estimation for the RFOU}\label{paraest}
{\color{black} In this section, we shall obtain exponential moment estimates for the RFOU process $X$ (Section \ref{esti}), then derive Girsanov theorem for $X$  and use it to derive the MLE of the drift parameter (Section \ref{MLE}), and finally establish the strong consistency and the asymptotic normality of the MLE (Section \ref{Prop}).}

\subsection{Estimates on the RFOU process}\label{esti}
For any $\beta  \in (0,1)$, we denote by $C^{\beta  }(0,T)$ the space of $%
\beta $-H\"{o}lder continuous functions on the interval $[0,T]$. For  $x\in C^{\beta  }(0,T)$, we
will
make use of the notations
\begin{equation}\label{eq5.1}
\left\| x\right\| _{\beta,T }=\sup_{0\leq \theta <r\leq T}\frac{%
|x_{r}-x_{\theta }|}{|r-\theta |^{\beta }},
\end{equation}
and%
\begin{equation}\label{eq5.2}
\Vert x||_{\infty,T }=\sup_{0\leq r\leq T}|x_{r}|\,.
\end{equation}

\begin{lem}\label{fernique}
Let $X_t$  be the strong solution to the following Skorohod equation
\beq\label{2190}dX_t=-\alpha X_tdt+\sigma dW^H_t+ dL_t\eeq subject to $X(0)=x\ge0$, where $L$ is the minimal non-decreasing process.
Then $X$ is $(H-)$ H\"{o}lder continuous.
There exists $\lambda_0>0$ depending on $T,H$ and $\epsilon$ such that
\beq\label{fe}\mathbb E \exp(\lambda_0\|X\|^2_{\infty,T})<\infty \quad \text{and} \quad \mathbb E \exp(\lambda_0 \|X\|_{H-\epsilon,T}^2)<\infty.\eeq
\end{lem}

\begin{proof}


\textcolor{black}{From the explicit representation of the local time $L$ of the one-dimensional Skorohod reflection problem (cf. \cite{Harrison-book, Whitt2002}), we have}
\[X_t=x-\alpha \int_0^t X_sds+\sigma W_t^H+\sup_{0\le s\le t}\{-x+\alpha\int_0^s X_udu-\sigma W_s^H\}\vee 0\]
\textcolor{black}{for each pathwise solution $X$ to \eqref{2190}.}

Assume $\alpha>0$. We have that
\beqys
X_t&\le& x-\alpha \int_0^t X_sds+\sigma W_t^H+\alpha\int_0^t X_sds+\sup_{0\le s\le t}\{-x-\sigma W_s^H\}\vee 0\\
&\le& 2x+2\sigma\|W^H\|_{\infty,T}.
\eeqys
For general $\alpha\in\mathbb R$, we have
\[
|X_t|\le 2x+2|\alpha| \int_0^t |X_s|ds+2\sigma \|W^H\|_{\infty,T}.
\]
By Gronwall's inequality, we have
\[\|X\|_{\infty,T}\le 2(x+\sigma\|W^H\|_{\infty,T})e^{2|\alpha| T}\]
Then, the first inequality in \eqref{fe} follows from the Fernique's Theorem, which claims the exponential integrability of the square of a seminorm for a Gaussian process (see \cite{MR0413238}).

For the second inequality, noting that  the function $f(x)=x\vee0:=\max\{x,0\}$ is Lipschitz, and that for a general continuous function $g$,
\beqys
\sup\limits_{0\le s\le t+r}\{g_s\}\vee0-\sup\limits_{0\le s\le t}\{g_s\}\vee0
&\le& \left(\sup_{t\le s\le t+r}\{g_s\}-\sup_{0\le s\le t}\{g_s\}\right)\vee0\\
&\le& \sup_{t\le s\le t+r}\{g_s\}-g_t\\
&\le &  \sup_{t\le s,u\le t+r}\{g_s-g_u\},
\eeqys
we have
\[|X_r-X_s|\le 2|\alpha|\|X\|_\infty(r-s)+2\sigma \|W\|_{H-\epsilon,T} (r-s)^{H-\epsilon}. \]
So we get that
\[\|X\|_{H-\epsilon,T}\le 2|\alpha|\|X\|_{\infty,T} T^{1+\epsilon-H}+2\sigma \|W\|_{H-\epsilon,T}.\]
Applying Fernique's Theorem once more, we have $\E e^{\lambda \|W\|_{H-\epsilon,T}^2}<\infty$ for some $\lambda>0$. Combining with the first inequality for $\|X\|_{\infty,T}$, we have the second inequality.
\end{proof}


\subsection{MLE based on the Girsanov transform}\label{MLE}
{\color{black} Let $0<s<t\le T$ and $H\in (0, \frac12) \cup (\frac12, 1)$. Recall} (from Section \ref{fbm}) the following notations
\begin{align*}
&k_H(t,s)=\kappa_H^{-1}s^{\frac12-H}(t-s)^{\frac12-H},\quad\quad \kappa_H=2H\Gamma(\frac32-H)\Gamma(H+\frac12);\\
&M_t^H=\int_0^t k_H(t,s)dW_s^H.
\end{align*}
The quadratic variation of the martingale $M^H$ is $\langle M^H \rangle_t=\lambda_H^{-1}t^{2-2H},$ where $\lambda_H=\dfrac{2H\Gamma(3-2H)\Gamma(H+\frac12)}{\Gamma(\frac32-H)}.$

{\color{black}The following Girsanov theorem for fractional Brownian motion was given in \cite{NVV99} based on the fundamental martingale $M^H$.
\begin{thm}[Girsanov theorem for shifted fractional Brownian motion]
For $a\in \mathbb R$ and $H\in (0, \frac12) \cup (\frac12, 1)$, $X_t=W^H_t +at, t\ge 0$ is a fractional Brownian motion with Hurst parameter $H$ under $P_a$ where $P_a$ is given by
\[\frac{d P_a}{dP}=\exp\left(-aM_t^H-\frac12 a^2\langle M^H\rangle_t\right).\]
\end{thm}}

Now we develop Girsanov theorem for our RFOU process $X$ by using the fundamental martingale $M^H$. Let \beq\label{de}\tilde X_t=\int_0^t k_H(t,s)dX_s, \quad \tilde L_t=\int_0^tk_H(t,s)dL_s, \quad \mbox{and} \quad \chi_t=\frac{d}{d\langle M^H \rangle_t} \int_0^t k_H(t,s)X_sds.\eeq Then we have
\beq\label{de1}d\tilde X_t=-\alpha \chi_t d\langle M^H \rangle_t+\sigma dM_t^H+d\tilde L_t.\eeq

\begin{thm}[Girsanov theorem for RFOU] \label{Girsanov2}
 For $H\in (0, \frac12) \cup (\frac12, 1)$, let
\[\eta_T=\exp\left(\frac\alpha\sigma \int_0^T\chi_sdM_s^H-\frac{\alpha^2}{2\sigma^2}\int_0^T\chi_s^2d\langle M^H \rangle_s\right).\]
Then $ \{\sigma^{-1}X_s:0\le s\le T\}$ is a reflected  fractional Brownian motion under the new probability $\tilde P$ defined by $d\tilde P/dP=\eta_T$.
\end{thm}
\begin{proof}{\color{black}
It suffices to prove that $\E(\eta_T)=1.$ Indeed, if $\E(\eta_T)=1$, then by (\ref{ee4.1}) and classical Girsanov theorem for standard Brownian motion (see, e.g., \cite{MR1800857}), we have that $\{\sigma M_t^H, 0\le t\le T\}$ under the original probability $P$ has the same distribution as $\{-\alpha\int_0^t \chi_s d\langle M^H\rangle_s +\sigma M_t^H, 0\le t\le T\}$ under the new probability $\tilde P$. Therefore $ \{\sigma^{-1}X_t:0\le t\le T\}$ is a reflected  fractional Brownian motion under $\tilde P$.

Now we show that $\E(\eta_T)=1.$ By Lemma \ref{quadvar} below and Lemma \ref{fernique}, there exists a positive constant $\delta$ such that for
\begin{equation}\label{eq5.6'}
\E \exp\left(\frac{\alpha^2}{2\sigma^2}\int_t^{t+\delta} \chi_s^2 d\langle M^H\rangle_s\right)<\infty
\end{equation}
 for $0\le t<t+\delta\le T$.

 Choose $n$ big enough such that $\dfrac Tn\le \delta $, and let $t_i=\dfrac{iT}{n}, i=1,\cdots, n.$
Denoting  $\frac{\alpha}{\sigma} \chi_s$ by $f(s)$, we have
\beqys
\E[\eta_T]&=&\E\exp\left(\int_0^Tf(s)dM^H_s-\frac12\int_0^Tf^2(s)d\langle M^H\rangle_s\right)\\
&=&\E\left(\E\left[\left.\exp\left(\int_0^Tf(s)dM^H_s-\frac12\int_0^Tf^2(s)d\langle M^H\rangle_s\right)\right|\mathcal F_{t_{n-1}}\right]\right)\\
&=&\E\bigg(\exp\left(\int_0^{t_{n-1}}f(s)dM^H_s-\frac12\int_0^{t_{n-1}}f^2(s)d\langle M^H\rangle_s\right)\\
&&\quad \times\E\left[\left.\exp\left(\int_{t_{n-1}}^Tf(s)dM^H_s-\frac12\int_{t_{n-1}}^Tf^2(s)d\langle M^H\rangle_s\right)\right|\mathcal F_{t_{n-1}}\right]\bigg).
\eeqys
We have that $\displaystyle\E\left[\left.\exp\left(\int_{t_{n-1}}^Tf(s)dM^H_s-\frac12\int_{t_{n-1}}^Tf^2(s)d\langle M^H\rangle_s\right)\right|\mathcal F_{t_{n-1}}\right]\le 1$ a.s., since\\ $\displaystyle\exp\left(\int_{0}^\cdot f(s)dM^H_s-\frac12\int_{0}^\cdot f^2(s)d\langle M^H\rangle_s\right)$ is a positive local martingale and hence a supermartingale. On the other hand, the Novikov's condition $ \displaystyle \E\left[\exp\left(\frac12\int_{t_{n-1}}^Tf^2(s)d\langle M^H\rangle_s\right)\right]<\infty$  is fulfilled by (\ref{eq5.6'}), and hence $\displaystyle\E\left[\exp\left(\int_{t_{n-1}}^T f(s)dM^H_s-\frac12\int_{t_{n-1}}^Tf^2(s)d\langle M^H \rangle_s\right)\right]= 1$. Therefore \[\E\left[\left.\exp\left(\int_{t_{n-1}}^T f(s)dM^H_s-\frac12\int_{t_{n-1}}^Tf^2(s)d\langle M^H\rangle_s\right)\right|\mathcal F_{t_{n-1}}\right]= 1 \quad \mbox{a.s.}\]
We may repeat the above procedure for $\displaystyle \E\exp\left(\int_0^{t_{i-1}}f(s)dM^H_s-\frac12\int_0^{t_{i-1}}f^2(s)d\langle M^H\rangle_s\right), i=n-1, n-2, \cdots, 2, 1$, and finally we obtain        \[\E(\eta_T)=\E\exp\left(\int_0^{t_{1}}f(s)dM^H_s-\frac12\int_0^{t_{1}}f^2(s)d\langle M^H\rangle_s\right)=1. \]}
\end{proof}

\begin{lem}\label{quadvar}
There exists a constant $C>0$ depending on $H, T$ and $\varepsilon\in(0,H\wedge \frac12)$ such that for $0\le a<b\le T$
\begin{equation}
\int_a^b \chi_s^2 d\langle M^H \rangle_s\le
\begin{cases} C \left[ \|X\|_{\infty,T}^2(b-a)^{2-2H}+\|X\|_{H-\varepsilon,T}^2 (b-a)\right], & H\in (\frac12 ,1);\\
C \|X\|_{\infty,T}^2 (b-a), & H\in (0,\frac12).
\end{cases}
\end{equation}
\end{lem}
\begin{proof}
We begin by noting that \[\chi_t=\kappa_H\frac{d}{dt}\left(\int_0^ts^{\frac12-H}(t-s)^{\frac12-H}X_sds\right)\frac{dt}{d\langle M^H \rangle_t}.\]

Let $\displaystyle Y_t=\int_0^t s^{\frac12-H}X_sds.$

{\bf Case 1:  $H>\dfrac12$.} We have that
\beqy
&&\int_0^t (t-s)^{\frac12-H}s^{\frac12-H}X_sds\notag\\
&=&\int_0^t (t-s)^{\frac12-H}d(Y_s-Y_t)\notag\\
&=&\lim_{s\uparrow t} (t-s)^{\frac12-H}(Y_s-Y_t)+t^{\frac12-H}Y_t+(H-\frac12)\int_0^t(Y_t-Y_s)(t-s)^{-\frac12-H}ds\notag\\
&=&t^{\frac12-H}Y_t+(H-\frac12)\int_0^t(Y_t-Y_s)(t-s)^{-\frac12-H}ds\notag\\
&=&t^{\frac12-H}Y_t+(H-\frac12)\int_0^t(Y_t-Y_{t-s})s^{-\frac12-H}ds,\label{e5.81}
\eeqy
{\color{black} where in the third equality we used the fact \[ 0\le \lim_{s\uparrow t} (t-s)^{\frac12-H} (Y_t-Y_s) = \lim_{s\uparrow t} (t-s)^{\frac12-H} \int_s^t r^{\frac12-H}X_rdr\le (\frac32-H)^{-1}\lim_{s\uparrow t} (t-s)^{\frac12-H} (t^{\frac32 -H}-s^{\frac32-H})\|X\|_{\infty,T}=0.\]}
Also, observe that
\beqy
&&\frac{d}{dt}\int_0^t(Y_t-Y_{t-s})s^{-\frac12-H}ds\notag\\
&=&\lim_{\varepsilon\to0}\frac{1}{\varepsilon}\left(\int_0^{t+\varepsilon}(Y_{t+\varepsilon}-Y_{t+\varepsilon-s})s^{-\frac12-H}ds-\int_0^t(Y_t-Y_{t-s})s^{-\frac12-H}ds\right)\notag\\
&=:&A(t)+B(t), \label{e5.82}
\eeqy
where $$A(t)=\lim_{\varepsilon\to0}\frac{1}{\varepsilon}\left(\int_0^{t+\varepsilon}(Y_{t+\varepsilon}-Y_{t+\varepsilon-s})s^{-\frac12-H}ds-\int_0^t(Y_{t+\varepsilon}-Y_{t+\varepsilon-s})s^{-\frac12-H}ds\right),
$$
and $$B(t)=\lim_{\varepsilon\to0}\frac{1}{\varepsilon}\left(\int_0^{t}(Y_{t+\varepsilon}-Y_{t+\varepsilon-s})s^{-\frac12-H}ds-\int_0^t(Y_t-Y_{t-s})s^{-\frac12-H}ds\right).
$$
It is clear that
\beq
A(t)=t^{-\frac12-H}Y_t, \label{e5.83}
\eeq when $ t>0$ since $Y_t$ is continuous in $t$.

For $B(t)$, we have
\beqys
B(t)&=&\int_0^t(Y_t'-Y_{t-s}')s^{-\frac12-H}ds\\
&=&\int_0^t\left(t^{\frac12-H}X_t-(t-s)^{\frac12-H}X_{t-s}\right)s^{-\frac12-H}ds\\
&=&t^{\frac12-H}\int_0^t\left(X_t-X_{t-s}\right)s^{-\frac12-H}ds+\int_0^tX_{t-s}\left(t^{\frac12-H}-(t-s)^{\frac12-H}\right)s^{-\frac12-H}ds.
\eeqys
Observe that \[ 0\le \int_0^t\left(X_t-X_{t-s}\right)s^{-\frac12-H}ds\le \int_0^t\|X\|_{H-\varepsilon,T}s^{H-\varepsilon}s^{-\frac12-H}ds=\dfrac{2\|X\|_{H-\varepsilon,T}}{1-2\varepsilon}t^{\frac12-\varepsilon},\]
 and also that
  \[\int_0^t\left|t^{\frac12-H}-(t-s)^{\frac12-H}\right|s^{-\frac12-H}ds=t^{1-2H}\int_0^1\left| 1-(1-s)^{\frac12-H}\right|s^{-\frac12-H}ds.\]
 {\color{black} Note that  $\displaystyle \int_0^1\left| 1-(1-s)^{\frac12-H}\right|s^{-\frac12-H}ds<\infty$, since the integrand function is continuous on $(0,1)$, while bounded by $2(1-s)^{\frac12-H}$ when $s$ is close to $1$, and by $ (1-s)^{-\frac12-H}s^{\frac12-H}$ when $s$ is close to $0$.} Therefore
 \beqy \label{e5.84}
 |B(t)|\le C\left(\|X\|_{H-\varepsilon,T}t^{1-H-\varepsilon}+\|X\|_{\infty,T}t^{1-2H}\right).
 \eeqy

Combining equations (\ref{e5.81}), (\ref{e5.82}), (\ref{e5.83}) and (\ref{e5.84}), we have
\beqys
\left|\frac{d}{dt}\int_0^t (t-s)^{\frac12-H}s^{\frac12-H}X_sds\right|&\le &C\left|t^{\frac12-H}t^{\frac12-H}X_t+t^{-\frac12-H}Y_t+B(t)\right|\\
&\le& C\left(\|X\|_{\infty,T}t^{1-2H}+\|X\|_{H-\varepsilon,T}t^{1-H-\varepsilon}\right).
\eeqys

Noting that $\dfrac{dt}{d\langle M^H\rangle_t}=\dfrac{\lambda_H}{2-2H}t^{2H-1}$, for $\chi_t$ we have
\[|\chi_t|\le C\left(\|X\|_{\infty,T}+\|X\|_{H-\varepsilon,T}t^{H-\varepsilon}\right).\]
Hence, we conclude that \[ \int_a^b \chi_s^2d\langle M^H \rangle_s=\lambda_H^{-1}(2-2H)\int_a^b\chi_s^2s^{1-2H}ds\le C\left(\|X\|_{\infty,T}^2(b-a)^{2-2H}+\|X\|_{H-\varepsilon,T}^2 T^{1-2\varepsilon}(b-a)\right).\]
{\bf Case 2: $H<\dfrac12$.}  Notice that
\beqys
\dfrac{d}{dt}\int_0^t (t-s)^{\frac12-H} s^{\frac12-H}X_sds
&=&\Gamma(\dfrac32-H)\dfrac{d}{dt} \left(I_{0+}^{\frac32-H}(s^{\frac12-H}X_s)(t)\right)\\
&=&\Gamma(\dfrac32-H)I_{0+}^{\frac12-H}(s^{\frac12-H}X_s)(t),
\eeqys
and also
\[\left|I_{0+}^{\frac12-H}(s^{\frac12-H}X_s)(t)\right|=\dfrac{1}{\Gamma(\frac12-H)}\left|\int_0^t (t-s)^{-\frac12-H}s^{\frac12-H}X_sds\right|\le C\|X\|_{\infty,T} t^{1-2H}.\]
Therefore, $|\chi_t|\le C\|X\|_{\infty,T}.$ Hence, we get
 \[ \int_a^b\chi_s^2d\langle M^H \rangle_s=\lambda_H^{-1}(2-2H)\int_a^b \chi_s^2s^{1-2H}ds\le C\|X\|_{\infty,T}^2T^{1-2H} (b-a).\]
\end{proof}

\begin{rem}\label{link}
From the proof  of the above lemma, we have that
\begin{itemize}
\item if {\color{black} $H>\dfrac12$},
\beqys \frac{d}{dt}\int_0^t (t-s)^{\frac12-H}s^{\frac12-H}X_sds
&=& t^{\frac12-H}Y_t+(H-\frac12)(A(t)+B(t))\\
&=& t^{1-2H}X_t+(H-\frac12)B(t)\\
&=& t^{1-2H}X_t+(H-\frac12)\int_0^t\left(t^{\frac12-H}X_t-(t-s)^{\frac12-H}X_{t-s}\right)s^{-\frac12-H}ds\\
&=&\Gamma(\frac32-H)D_{0+}^{H-\frac12}(u^{\frac12-H}X_u)(t), \eeqys and

\item if {\color{black}$H<\dfrac12$},
\beqs
\dfrac{d}{dt}\int_0^t (t-s)^{\frac12-H} s^{\frac12-H}X_sds = \Gamma(\dfrac32-H)I_{0+}^{\frac12-H}(u^{\frac12-H}X_u)(t).
\eeqs
\end{itemize}
If we define $D_{0+}^\alpha$ as $I_{0+}^{-\alpha}$ for $\alpha<0,$ then for all $H\in (0, \frac12) \cup (\frac12, 1)$, $\chi_t$ can be represented as
\begin{equation}\label{e.chi}
\chi_t=\frac{\Gamma(\frac32-H)\lambda_H}{2-2H}\kappa_H^{-1} t^{2H-1}D_{0+}^{H-\frac12}(u^{\frac12-H}X_u)(t).
\end{equation}
\end{rem}
{\color{black} \begin{rem}\label{r5.6}
By equation (\ref{e.Koperator}), we have
\begin{equation*}
(K_H^{-1}\varphi )(t)=
\begin{cases} \dfrac{1}{C_H\Gamma(H-\frac12)} t^{H-\frac12}D_{0+}^{H-\frac12}(u^{\frac12-H}\varphi')(t), &\text{ if }  H>\frac12\\
\dfrac{1}{b_H\Gamma(H+\frac12)} t^{H-\frac12}I_{0+}^{\frac12-H}(u^{\frac12-H}\varphi')(t), &\text{ if } H<\frac12.
\end{cases}
\end{equation*}
Noting that $ C_H=b_H(H-\frac12)$ when $H>\frac12$, and using the convention $D_{0+}^\alpha:=I_{0+}^{-\alpha}$  for  $\alpha<0$, we have a uniform representation for $K_H^{-1}:$
\begin{equation}\label{e.Kinverse}
(K_H^{-1}\varphi )(t)=\dfrac{1}{b_H\Gamma(H+\frac12)} t^{H-\frac12}D_{0+}^{H-\frac12}(u^{\frac12-H}\varphi')(t),\, H\in (0, \frac12) \cup (\frac12, 1).
\end{equation}
Using  (\ref{e.chi}), (\ref{e.Kinverse}) and (\ref{ee4.1}), by direct computations we can get
\[\int_0^T \chi_t dM_t^H=\int_0^T (K_H^{-1}\int_0^\cdot X_udu)(s)dB_s,\]
where $B$ is the Wiener process defined in (\ref{ee4.1}). By Theorem \ref{connection}, we know that $B$ coincides with $W$, the Wiener process related with the fractional Brownian motion $W^H$. Therefore, \textcolor{black}{we conclude that} Theorem \ref{Girsanov2} has the following equivalent version.
\end{rem}}

\begin{rem}\label{infivar} We note that the infinitesimal variance parameter $\sigma^2$ of RFOU process \eqref{0.1}  can be estimated by using the quadratic variation of $\tilde X$ since $\langle \tilde X\rangle_t=\sigma^2 \langle M^H\rangle_t$ for $t>0$  by (\ref{de1}), where $\tilde X$ given in (\ref{de}) is a functional of $X$ and $M^H$ defined in (\ref{e4.0}) is the fundamental martingale.\end{rem}

\begin{thm}[Girsanov theorem for RFOU] \label{Girsanov1}
For $H\in(0,\frac12)\cup (\frac12, 1)$, let
\[\xi_T=\exp\left(\int_0^T\left(K_H^{-1}\int_0^\cdot \frac\alpha\sigma X_r dr\right)(s)dW_s-\frac12\int_0^T\left(K_H^{-1}\int_0^\cdot \frac{\alpha}\sigma X_r dr \right)^2(s)ds\right).\]
Then $ \{\sigma^{-1}X_s:0\le s\le T\}$ is a reflected fractional Brownian motion
under the new probability $\wh P$ defined by $d\wh P/dP=\xi_T$.
\end{thm}
In fact, $\xi_T=\eta_T$ a.s., and hence $\tilde P=\wh P$.
In the Appendix, we shall provide a direct proof of Theorem \ref{Girsanov1} by using fractional calculus.

Let $P^\alpha$ and $P^R$ denote the probability measures on $(C([0,T]),\mathcal B_T)$ induced by $X$  and $X^R$ respectively, where $X^R$ is the reflected fractional Brownian motions (i.e. $\alpha=0$). Similar to the proof of  \cite[Theorem 7.1]{MR1800857}, using Theorem \ref{Girsanov2}, we can show that $P^\alpha\sim P^R$, and
\beqys
\frac{dP^\alpha}{dP^R}=\eta_T^{-1}&=&\exp\left(-\frac\alpha\sigma \int_0^T\chi_sdM_s^H+\frac{\alpha^2}{2\sigma^2}\int_0^T\chi_s^2d\langle M^H \rangle_s\right)\\
&=&\exp\left(-\frac\alpha{\sigma^2}\int_0^T\chi_sd\tilde X_s-\frac{\alpha^2}{2\sigma^2}\int_0^T\chi_s^2d\langle M^H \rangle_s+\frac{\alpha}{\sigma^2}\int_0^T \chi_sd\tilde L_s\right).
\eeqys
Hence the MLE for $\alpha$ is given by
\beq\label{es}\tilde\alpha_T:=\dfrac{-\int_0^T \chi_sd\tilde X_s+\int_0^T \chi_s d\tilde L_s}{\int_0^T \chi_s^2 d\langle M^H \rangle_s}.\eeq
{\color{black}\begin{rem} Although Theorem \ref{Girsanov1} is equivalent to Theorem \ref{Girsanov2}, it is not practical to get MLE for $\alpha$ by applying Theorem \ref{Girsanov1} directly as above. This is because, to get an estimator just involving the information of $X$ and $L$, one needs to transform $\int_0^T \left( K_H^{-1}\int_0^\cdot \frac\alpha\sigma X_r dr\right)(s)dW_s$ to $\int_0^T \left((K_H^*)^{-1} K_H^{-1}\int_0^\cdot \frac{\alpha}{\sigma}X_rdr\right)(s)dW_s^H$ (and then use $X$ and $L$ to represent $W^H$). The integral $\int_0^T \left((K_H^*)^{-1} K_H^{-1}\int_0^\cdot \frac{\alpha}{\sigma}X_rdr\right)(s)dW_s^H$ is in Skorohod sense but not in Stratonovich sense, and hence if we replace $dW^H_s$ by $\frac1\sigma dX_s+\frac\alpha\sigma X_sds-\frac1\sigma dL_s$ in the integral, we have to deal with the computations of Skorohod integrals against $X$ and $L$, which is rather complex and \textcolor{black}{im}practical for an estimator.
\end{rem}}

\subsection{Properties of the MLE}\label{Prop}
The following lemma will play an essential role in establishing the strong consistency of the MLE $\tilde \alpha_T$ in Theorem \ref{prop4.2}.

\begin{lem}\label{fou}
For $H\in (0, \frac12) \cup (\frac12, 1)$, let $\{Y_t, t\geq0\}$ be the fractional Ornstein-Uhlenbeck (FOU) process satisfying $$dY_t=-\alpha Y_t dt+\sigma dW_t^H$$ with initial value $Y_0=x>0$.  Then we have {\color{black} $P\{X_t \ge |Y_t|,t\ge 0\}=1 $}. Furthermore, we also have
\begin{align*}
&\lim_{T\to\infty} \dfrac1T\int_0^T |Y_t|dt=\E\left[\left|\si\int_{-\infty}^0 e^{-\alpha(t-s)}dW_s^H+x\right|\right],{\color{black} \quad a.s.,}&\hbox{ when } \alpha>0,\\
&\lim_{T\to\infty} \dfrac1T\int_0^T |Y_t|dt=\infty, {\color{black} \quad a.s.,} & \hbox{ when } \alpha\le 0.
\end{align*}
\end{lem}
\begin{proof}
Since $d(X_t-Y_t)=-\alpha(X_t-Y_t)dt+dL_t$ with $X_0-Y_0=0$, we have $X_t\ge Y_t$ a.s. for all $t$.
Let $Z_t$ be the  FOU process satisfying $dZ_t=-\alpha Z_t dt-\sigma dW_t^H$ with initial value $Z_0=-x$. Then we have $Z_t=-Y_t.$ On the other hand, $d(X_t+Z_t)=-\alpha(X_t+Z_t)dt+dL_t$ with $X_0+Z_0=0$, and hence $-X_t\le Z_t=-Y_t.$ So we have $|Y_t|\le X_t$ a.s. for all $t$. {\color{black} By the countability of the set $\mathbb Q^+$ of positive rational numbers, we have that $P\{X_t \ge |Y_t|,t\in\mathbb Q^+\}=1 $, and $P\{X_t \ge |Y_t|,t\ge 0\}=1 $ just follows from the fact that both $X$ and $Y$ have continuous trajectories.}

Note that $Y_t$ has the following expression,
 \[ Y_t=x+\si\int_0^t e^{-\alpha(t-s)}dW_s^H=x+\si\int_{-\infty}^t e^{-\alpha(t-s)}dW_s^H-e^{-\alpha t}\si\int_{-\infty}^0 e^{\alpha s}dW_s^H.\]

 When $\alpha>0,$ the process $\displaystyle\left\{\tilde Y_t=\si\int_{-\infty}^t e^{-\alpha(t-s)}dW_s^H, t\geq0 \right\}$ is Gaussian, stationary and ergodic (see \cite{MR1961165}), and by ergodic theorem,
\[\lim_{T\to\infty}\frac1T\int_0^T |\tilde Y_t+x| dt=\E(|\tilde Y_0+x|),\]
which implies that
\[\lim_{T\to\infty}\frac1T\int_0^T | Y_t| dt=\E(|\tilde Y_0+x|)>0.\]

{\color{black} When $\alpha=0$, $Y_t=x+\si W_t^H$. It suffices to show
\begin{equation}\label{e5.13'}
\lim_{T\to\infty}\frac1T  \int_0^T |W_t^H| dt=\infty, \quad a.s.,
\end{equation}
which follows from Lemma \ref{ergodicity} in the Appendix.}

When $\alpha<0,$ let $\beta=-\alpha>0,$ and then
\[Y_t=x+\si\int_0^t e^{\beta(t-s)}dW_s^H=x+\si\int_0^{\infty} e^{\beta(t-s)}dW_s^H-\si\int_t^{\infty} e^{\beta(t-s)}dW_s^H.\]

Similarly as in \cite{MR1961165}, we can show that  $\displaystyle\left\{\hat Y_t=\si\int_t^{\infty} e^{\beta(t-s)}dW_s^H, t\ge0\right\}$ is stationary and ergodic. Note that now $Y_t=x+e^{\beta t}\hat Y_0-\hat Y_t$ and hence
\beqys \liminf_{T\to\infty} \frac1T\int_0^T |Y_t|dt
&\ge&  |\hat Y_0| \lim_{T\to\infty}\frac1T\int_0^T e^{\beta t}dt -\lim_{T\to\infty}\frac1T\int_0^T |x-\hat Y_t|dt\\
&\ge& \infty \quad\hbox{a.s. }
\eeqys
\end{proof}
\begin{thm}\label{prop4.2}
For $H\in (0, \frac12) \cup (\frac12, 1)$, the MLE $\tilde \alpha_T$ in \eqref{es} is strongly consistent.
\end{thm}
\begin{proof}
To prove $\lim\limits_{T\to\infty} \tilde \alpha_T=\alpha$ a.s., it suffices to show that
\[\lim_{T\to\infty} \frac{\int_0^T \chi_sdM_s^H}{\int_0^T \chi_s^2d\langle M^H \rangle_s }=0 {\color{black} \quad a.s.}\]
Since $\displaystyle\int_0^T \chi_sdM_s^H$ is a martingale, by Lepingle's law of large numbers (see \cite[Lemma 17.4]{MR1800857}), it suffices to show that
\begin{equation}\label{e.4.2}
\lim_{T\to\infty}\int_0^T \chi_s^2d\langle M^H \rangle_s =\infty {\color{black} \quad a.s.}
\end{equation}
{\color{black} From Remark \ref{link}, we know that for $H\in (0, \frac12) \cup (\frac12, 1),$  \[\chi_t= Ct^{2H-1}D_{0+}^{H-\frac12}(u^{\frac12-H}X_u)(t),\]
for some constant $C$ depending on $H$ only.
Note that $D_{0+}^{H-\frac12}$  means $I_{0+}^{\frac12-H}$ when $H<\frac12.$}

Let $c=1-2H$ and $d=4H-2$. Then $c+d=2H-1$ and $c+\frac d2=0$. Notice that
\beqys \int_0^T \chi_s^2d\langle M^H\rangle_s  &=&C\int_0^T s^{2H-1}\left(D_{0+}^{H-\frac12}u^{\frac12-H}X_u\right)^2(s)ds\\
&=&C\int_0^T s^{c}\left(s^{\frac d2}\left(D_{0+}^{H-\frac12}u^{\frac12-H}X_u\right)(s)\right)^2ds\\
&\ge & C \left[\frac1{T^{\frac {1+c}{2}}}\int_0^T s^{c+\frac d2}\left(D_{0+}^{H-\frac12}u^{\frac12-H}X_u\right)(s)ds\right]^2\\
&= & C \left[\frac1{T^{\frac {1+c}{2}}}\int_0^T \left(D_{T-}^{H-\frac12}1\right)(s)s^{\frac12-H}X_sds\right]^2\\
&= & C \left[\frac1{T^{1-H}}\int_0^T (T-s)^{\frac12-H}s^{\frac12-H}X_sds\right]^2,
\eeqys
where the only inequality above follows by Jensen's inequality or H\"older's inequality, {\color{black} and the third equality follows by equation (\ref{ibp})}.

If {\color{black} $H> \dfrac12$}, we have
\[\left[\frac1{T^{1-H}}\int_0^T (T-s)^{\frac12-H}s^{\frac12-H}X_sds\right]^2=\left[\frac1{T^{H}} \int_0^T (1-\frac sT)^{\frac12-H}\left(\frac sT\right)^{\frac12-H}X_sds\right]^2 \ge  C  \left[\frac1{T^{H}}\int_0^T X_sds\right]^2.\]

Therefore, to prove (\ref{e.4.2}), it suffices to show that
\begin{equation}\label{e.4.4}
{\color{black}\liminf_{T\to\infty}}\frac1T\int_0^T X_sds>0 {\color{black} \quad a.s.},
\end{equation}
and it is an immediate consequence of Lemma \ref{fou}.

{\color{black} If $H< \dfrac12$,} we can show that
\begin{equation}\label{eq6.6}
\liminf_{T\to\infty} \frac1{T^{\frac32-H}} \int_0^T (T-s)^{\frac12-H}s^{\frac12-H}X_sds>0 {\color{black} \quad a.s.},
\end{equation}
which is a sufficient condition for  inequality (\ref{e.4.2}).

In fact, we have
\begin{align*}
&\int_0^T (T-s)^{\frac12-H}s^{\frac12-H}X_sds\\
=\frac{2}{1-2H}&\int_0^T\left(\int_s^T(T-r)^{-\frac12-H}dr\right) s^{\frac12-H}X_sds\\
=\frac{2}{1-2H}&\int_0^T\left(\int_0^rs^{\frac12-H}X_sds\right) (T-r)^{-\frac12-H}dr.
\end{align*}
For the term $\displaystyle\int_0^rs^{\frac12-H}X_sds,$ we have
\[\liminf_{r\to\infty} \frac1r\int_0^r s^{\frac12-H}X_sds\ge\liminf_{r\to\infty} \frac1r\int_1^r s^{\frac12-H}X_sds\ge \liminf_{r\to\infty} \frac1r\int_1^r X_sds={\color{black} \liminf_{r\to\infty}} \frac1r\int_0^r X_sds\quad {\color{black} a.s. }\]
{\color{black} By Lemma \ref{fou}, there exists a positive number $a$ such that $\liminf_{r\to\infty} \frac1r\int_0^r X_sds\ge a$  a.s. Choose $\varepsilon\in (0,a)$,  then for almost all $\omega\in \Omega$, there exists $N(\omega)\in (0,\infty)$, such that $\displaystyle \frac1r\int_0^r s^{\frac12-H}X_s(\omega)ds \ge a-\varepsilon$ for all $r>N(\omega).$}

Therefore, for almost all $\omega\in \Omega$,
\begin{align*}
&\liminf_{T\to\infty} \frac1{T^{\frac32-H}} \int_0^T (T-s)^{\frac12-H}s^{\frac12-H}X_s(\omega)ds\\
= & \frac{2}{1-2H}\liminf_{T\to\infty} \frac1{T^{\frac32-H}} \int_{N(\omega)}^T\left(\int_0^rs^{\frac12-H}X_s(\omega)ds\right) (T-r)^{-\frac12-H}dr\\
\ge & \frac{2}{1-2H}\liminf_{T\to\infty} \frac1{T^{\frac32-H}} \int_{N(\omega)}^T(a-\varepsilon)r (T-r)^{-\frac12-H}dr\\
=&\frac{2}{1-2H}\beta(1,\frac12-H)(a-\varepsilon),
\end{align*}
which implies (\ref{eq6.6}).

\end{proof}

\begin{thm}\label{mle}
For $H\in (0, \frac12) \cup (\frac12, 1)$, the MLE $\tilde \alpha_T$ of $\alpha$ admits the asymptotic normality, i.e.,
\[\dfrac{\tilde \alpha_T-\alpha}{\sigma}\sqrt{\int_0^T \chi_s^2d\langle M^H \rangle_s}\overset{\mathcal L}{\longrightarrow} N(0,1) \quad \text{as}\quad T\rightarrow\infty. \]
\end{thm}
\begin{proof}
This is an immediate consequence of the Central Limit Theorem for martingales (see, for instance, \cite[Theorem B.10]{MR1717690}.)
\end{proof}


\setcounter{equation}{0}
\section{Sequential MLE}\label{seq}
Recall the strong consistency and asymptotic normality of the MLE $\tilde \alpha_T$, established in Theorem \ref{prop4.2} and Theorem \ref{mle}.  Such results are valuable from statistical analysis viewpoint of applications, however, \textcolor{black}{there are some drawbacks. Generally speaking,} the MLE is a biased estimator and its mean squared error (MSE) depends on the parameter to be estimated (Theorem \ref{mle}). Therefore, useful estimates for the bias and the MSE are not available (or very difficult to derive), and as a consequence, there is no guarantee that the bias (or, variance) decays  fast enough (as $T\rightarrow\infty$) to achieve the Cramer-Rao bound. {\color{black} Hence, we are unable to verify whether the classical Cramer-Rao lower bound can be attained or not for the MLE $\tilde \alpha_T$.}
To overcome such limitations, we consider the sequential estimation plan and verify that the proposed plan is significantly helpful both in asymptotic and non-asymptotic short time observation.

In contrast to the MLE, the proposed sequential maximum likelihood estimator (SMLE) is unbiased, exactly normally distributed (on the finite time observation), and its MSE has an explicit, simple expression that does not depend on the parameter to be estimated (see Theorem \ref{thm1} below).  The SMLE is uniformly normally distributed over the entire parameter space which is the real line.  Such results would be of ample use in applications to several areas such as engineering, financial and biological modeling where unknown parameter estimation is based on relatively shorter time observation.  Furthermore, an analog of the Cramer-Rao lower bound is proved and the SMLE is shown to be efficient among all unbiased estimation plans in the mean squared error sense (see Theorem \ref{thm2} below).

Define the stopping  time $\tau^H(h)$ as \beq\label{1.1.1} \tau^H(h):=\inf\left\{t\geq0: \inte{0}{t} \atob{\chi}{2}_sd\langle M^H \rangle_s\geq h \right\}, \quad 0<h<\infty. \eeq Then, $\atob{\mathcal F}{X}_{\tau^H(h)}$-measurable function $\wh \alpha_{\tau^H(h)}$ \beq\label{10.2}\wh \alpha_{\tau^H(h)}:=\frac{1}{h}\left[\inte{0}{\tau^H(h)}\chi_s d\tilde L_s - \inte{0}{\tau^H(h)}\chi_s d\tilde X_s\right] \eeq  is a sequential estimator.  From \eqref{de}--\eqref{de1}, it can be seen that \beq\label{1.2.1} \wh \alpha_{\tau^H(h)}=\alpha - \frac{\sigma}{h} \inte{0}{\tau^H(h)}\chi_s dM_s^H. \eeq

A proof of the next theorem follows along the similar lines of \cite{LipShir2001v2} by accommodating our basic model assumptions involving the reflection (i.e., state space) constraint and fractional Brownian noise with $H\in(0,\frac12) \cup(\frac12,1)$. In what follows, the index $\alpha$ in $\mathbb P$ and $\mathbb E$ emphasizes the fact that the distribution of the  state process is being considered for the prescribed value $\alpha$.
\begin{thm}\label{thm1}
For $H\in(0,\frac12) \cup(\frac12,1)$, the sequential estimation plan $(\tau^H(h),\wh \alpha_{\tau^H(h)})$, $0<h<\infty$, has the following properties: \bi \item[(i)] $\mathbb P_{\alpha}(\tau^H(h)<\infty)=1$ for all $\alpha\in(-\infty,\infty)$, \item[(ii)] $\mathbb E_\alpha(\wh \alpha_{\tau^H(h)}) =\alpha$ for all $\alpha\in(-\infty,\infty)$, \item[(iii)] $\mathbb E_\alpha(\wh \alpha_{\tau^H(h)}-\alpha)^2=\dfrac{\atob{\sigma}{2}}{h}$,
\item[(iv)] $\wh \alpha_{\tau^H(h)}-\alpha \overset{\mathcal L}{=} N(0,\dfrac{\sigma^2}{h})$,
\item[(v)] $\wh \alpha_{\tau^H(h)}\ra\alpha$ a.s. as $h\ra\infty$.  \ei
\end{thm}
\begin{proof}

We first show that the $\atob{\mathcal F}{X}_{\tau^H(h)}$-measurable random variable $\wh\alpha_{\tau^H(h)}$ is indeed the SMLE as follows. Consider an arbitrary stopping time $\tau$ with respect to the filtration $\{\atob{\mathcal F}{X}_t\}_{t\geq0}$ generated by the process $X$. Similar to \cite[Theorem 7.1]{MR1800857}, the probability measures $\atob{\mathbb P}{\theta}_{\tau, X}$ and $\atob{\mathbb P}{\alpha}_{\tau, X}$ induced by the processes stopped at time $\tau$
\beqs\label{2.0} \left. \begin{gathered} dX_t^\theta=-\theta X_t^\theta dt + \sigma dW^H_t + dL_t^\theta, \quad X_t^\theta \geq \textcolor{black}{0} \quad \mbox{for all } t\geq0,  \\  dX_t^\alpha=-\alpha X_t^\alpha dt + \sigma dW^H_t + dL_t^\alpha, \quad X_t^\alpha \geq \textcolor{black}{0} \quad \mbox{for all } t\geq0,  \end{gathered}  \\ \right\} \eeqs respectively, are equivalent and their Radon-Nikodym derivative is given by

  \beq\label{2.1}  \left.\frac{d\atob{\mathbb P}{\alpha}_{\tau,X}}{d\atob{\mathbb P}{\theta}_{\tau,X}}\right|_{\atob{\mathcal F}{X}_{\tau,\theta}} = \exp\left\{-\frac{1}{\sigma} \inte{0}{\tau}(\alpha-\theta)\atob{\chi}{\alpha}_tdM_t^H+\frac{1}{2\atob{\sigma}{2}}\inte{0}{\tau}(\alpha-\theta)^2(\atob{\chi}{\alpha}_t)^2{\color{black}d\langle M^H \rangle _t} \right\}, \eeq
 where $\atob{\mathcal F}{X}_{\tau,\theta}$ is the natural filtration generated by $\{\atob{X}{\theta}_t:0\leq t\leq \tau\}$, and \[\chi^\alpha:= \displaystyle\frac{d}{d\langle M^H \rangle_t} \int_0^t k_H(t,s)X_s^\alpha ds.\]

Let $\atob{X}{0}_t$ be the reflected fractional Brownian motion (RFBM) satisfying $d\atob{X}{0}_t=\sigma dW^H_t+dL_t$, $t\geq0$ and $\atob{\mathbb P}{0}_\tau$ be the measure induced by the RFBM $\atob{X}{0}$.  Then the log likelihood function $\ell_\tau(\alpha)$ is given by
 \[{\ell}_\tau(\alpha):=\sigma^2\log\frac{d\atob{\mathbb P}{\alpha}_{\tau,X}}{d\atob{\mathbb P}{0}_{\tau,X}} = -\alpha \inte{0}{\tau} \chi_t^\alpha d\tilde X_t -\frac{\atob{\alpha}{2}}{2}\inte{0}{\tau}\atob{(\chi^\alpha_t)}{2}d\langle M^H \rangle_t + \alpha \inte{0}{\tau}\chi_t d\tilde L_t. \]

 Then, by solving the equation $$\dot{\ell}_\tau(\alpha):=\frac{d}{d\alpha}\left(\sigma^2\log\frac{d\atob{\mathbb P}{\alpha}_{\tau,X}}{d\atob{\mathbb P}{0}_{\tau,X}}\right)=0,$$
 we obtain  the SMLE given by
 \beq\label{2.2}
  \wh\alpha_{\tau} = \frac{-\inte{0}{\tau}\chi_s^\alpha d\tilde X_s + \inte{0}{\tau}\chi_s^\alpha d\tilde L_s}{\inte{0}{\tau}\atob{(\chi_s^\alpha)}{2}d\langle M^H \rangle_s}=\alpha- \frac{\sigma \inte{0}{\tau}\chi_s^\alpha d M^H_s }{\inte{0}{\tau}\atob{(\chi_s^\alpha)}{2}d\langle M^H \rangle_s}.
 \eeq
Now, setting $\tau=\tau^H(h)$ in \eqref{2.2}, we obtain for $\wh\alpha_\tau=\wh\alpha_{\tau^H(h)}$ the representation given by \eqref{1.2.1}.

We shall use $\chi$ instead of $\chi^\alpha.$ To verify part (i), it suffices to show that $\lim\limits_{T\to\infty}\int_0^T \chi_s^2d\langle M^H\rangle_s=\infty$ a.s., which has already been proved in the proof of  Theorem \ref{prop4.2}.  The claims in (ii), (iii) and (iv) follow from the fact that the process $(\inte{0}{\tau^H(h)}\chi_s dM^H_s,h\ge0) $ is a standard Brownian motion indexed by $h$ (see, for instance, \cite[Theorem 4.6]{MR1121940}). Also, the result (v) is an immediate consequence of (iv).
\end{proof}

Next, we show that the proposed sequential estimation plan $(\tau^H(h),\wh\alpha_{\tau^H(h)})$ is efficient among all unbiased estimation plans in the mean squared error sense. \textcolor{black}{More precisely, we prove an analog of the Cramer-Rao lower bound for arbitrary unbiased estimation plans.}  
\begin{thm}\label{thm2}
Let the sequential plan $(\tau,\wh \alpha_\tau(X))$ be an arbitrary unbiased estimation plan for the RFOU process $\{X_t\}$ with the parameter $\alpha\in(-\infty,\infty)$, namely, \beq\label{4.1} \mathbb E_{\alpha}(\wh \alpha_\tau(X)) =\alpha \quad \mbox{for all}\quad \alpha\in(-\infty,\infty). \eeq Suppose also that $0<\mathbb E_\alpha[\inte{0}{\tau}\chi_s^2 d\langle M^H \rangle_s]<\infty$.  Then,  \beq\label{4.2} \var_\alpha(\wh\alpha_\tau)=\mathbb E_\alpha[\wh\alpha_\tau-\alpha]^2 \geq \frac{\atob{\sigma}{2}}{\mathbb E_\alpha[\inte{0}{\tau}\chi_s^2d\langle M^H \rangle_s]}.\eeq
\end{thm}

\begin{rem}
A sequential  estimation plan $(\tau,\wh\alpha_\tau)$ is said to be \emph{efficient} in the MSE sense if for which \eqref{4.2}  becomes an equality for all $\alpha\in(-\infty,\infty)$.  Notice that since $\mathbb E_\alpha[\wh\alpha_{\tau^H(h)}-\alpha]^2 =\frac{\atob{\sigma}{2}}{h}$ as established in Theorem \ref{thm1} (iii) and $\mathbb E_\alpha[\inte{0}{\tau^H(h)}\chi_s^2 d\langle M^H \rangle_s]=h$ by the definition of $\tau^H(h)$, the sequential estimation plan $(\tau^H(h),\wh \alpha_{\tau^H(h)})$ is efficient in the MSE sense.
\end{rem}

\begin{proof}
Without loss of generality, we assume that $\sigma=1$.  In view of the Radon-Nikodym derivative expression in \eqref{2.1} with $\theta=0$, differentiating both sides of \eqref{4.1} with respect to $\alpha$ yields that \beq\label{5.1} \mathbb E_\alpha \left[ \wh\alpha_\tau\left\{-\inte{0}{\tau} \chi_s d\tilde X_s-\alpha \inte{0}{\tau} \chi_s^2 d\langle M^H \rangle_s +\int_0^\tau \chi_s d\tilde L_s\right\} \right] =1 \eeq  (cf. the proof of Theorem 7.22 in \cite{LipShir2001v2}). Then, since \beqys\mathbb E_\alpha \left[\inte{0}{\tau} \chi_s d\tilde X_s+\alpha \inte{0}{\tau} \chi_s^2 d\langle M^H \rangle_s-\int_0^\tau \chi_s d\tilde L_s \right] &=& \mathbb E_\alpha \left[ \inte{0}{\tau}\chi_s d M^H_s \right]\\ &=&0, \eeqys  it follows that \beq\label{cs} \mathbb E_\alpha\left[(\wh \alpha_\tau-\alpha)\left(-\inte{0}{\tau} \chi_s d\tilde X_s-\alpha \inte{0}{\tau} \chi_s^2 d\langle M^H \rangle_s+\int_0^\tau \chi_s d\tilde L_s\right)\right]=1.\eeq Applying Cauchy-Schwarz inequality in \eqref{cs}, we obtain \beqy\label{12} 1&\leq& \mathbb E_\alpha[\wh\alpha_\tau-\alpha]^2\mathbb E_\alpha\left[\left(\inte{0}{\tau} \chi_s d\tilde X_s+\alpha \inte{0}{\tau} \chi_s^2 d\langle M^H \rangle_s-\int_0^\tau \chi_s d\tilde L_s\right)^2\right]\nonumber \\ &=&  \mathbb E_\alpha[\wh\alpha_\tau-\alpha]^2\mathbb E_\alpha\left[\left(\inte{0}{\tau}\chi_sdM_s^H\right)^2\right]\nonumber \\ &=&  \mathbb E_\alpha[\wh\alpha_\tau-\alpha]^2\mathbb E_\alpha\left[\inte{0}{\tau}\chi_s^2d\langle M^H \rangle_s\right], \eeqy where the first equality  follows  from the state equation \eqref{0.1}.  Since $0<\mathbb E_\alpha[\inte{0}{\tau}\chi_s^2d\langle M^H \rangle_s]<\infty$, we can divide both sides of \eqref{12} by this factor, and then the desired result follows.
\end{proof}

\setcounter{equation}{0}
\section{Appendix}
The following Girsanov theorem for the fractional Brownian motion was given in \cite{MR1934157}.

\begin{theorem}\label{girsanov}
Let $\displaystyle \tilde W^H_t= W^H_t+\int_0^t u_s ds$, where $\{u_t, t\in[0,T]\}$ is an adapted process with respect to the filtration $\mathcal F^{W^H}$ generated by $W^H$ with integrable trajectories.
Denote
\[\xi_T=\exp\left(-\int_0^T\left(K_H^{-1}\int_0^\cdot u_rdr\right)(s)dW_s-\frac12\int_0^T\left(K_H^{-1}\int_0^\cdot u_rdr \right)^2(s)ds\right),\]
where $W$ is the Wiener Process defined by (\ref{bm}).   Assume that
\begin{itemize}
\item[(i)] $\displaystyle\int_0^\cdot u_sds\in I_{0^+}^{H+\frac12}(L^2([0,T]))$, almost surely.
\item[(ii)] $\mathbb E(\xi_T)=1$.
\end{itemize}
Then the shifted process $\tilde W^H$ is an $\mathcal F_t^{W^H}$-fractional Brownian motion with Hurst parameter $H$ under the new probability $\tilde P$ defined by $d\tilde P/dP=\xi_T.$
\end{theorem}

We shall use the above Girsanov theorem and fractional calculus to provide an independent proof for Theorem \ref{Girsanov1}.

{\it Proof of  Theorem \ref{Girsanov1}.}
 If we can show that $\E(\xi_T)=1,$ then by Theorem \ref{girsanov},  $\displaystyle\wh W_t^H=W_t^H-\frac\alpha\sigma\int_0^t X_sds$ is a fractional Brownian motion under $\wh P$. Hence $\sigma^{-1}X_t=\sigma^{-1}X_0+ {\color{black} \wh W_t^H}+\sigma^{-1}L_t$ is a reflected fractional Brownian motion under $\wh P$.
We now show that for $\Delta t$ small enough, we have
\begin{equation}\label{novikov}
\mathbb E \exp\left(\int_t^{t+\Delta t}\left(K_H^{-1}\int_0^\cdot \lambda_0 X_r dr \right)^2(s)ds\right)<\infty.
\end{equation}

We shall prove this in two cases $H<\frac12$ and $H>\frac12$, respectively.

Case $H<\frac12$: Observe that
\beqys
\left|\left(K_H^{-1}\int_0^\cdot  X_rdr\right)(s)\right|
&=&\left|s^{H-\frac12}I_{0^+}^{\frac12-H}s^{\frac12-H}X_s\right|\\
&=&Cs^{H-\frac12}\left|\int_0^s (s-r)^{-\frac12-H}r^{\frac12-H}X_rdr\right|\\
&\le& C(1+|x|+\|X\|_{\infty,T}).
\eeqys
Then when $\Delta t$ small enough, (\ref{novikov}) follows by Lemma \ref{fernique}.

Case $H>\frac12$: We have that
\beqys
&&\left(K_H^{-1}\int_0^\cdot  X_rdr\right)(s)\\
&=&s^{\frac12-H}X_s+(H-\frac12)s^{H-\frac12}\left(X_s\int_0^s \frac{s^{\frac12-H}-r^{\frac12-H}}{(s-r)^{\frac12+H}}dr+\int_0^s \frac{X_s-X_r}{(s-r)^{\frac12+H}}r^{\frac12-H}dr\right).
\eeqys
Noting that
\[\int_0^s \frac{s^{\frac12-H}-r^{\frac12-H}}{(s-r)^{\frac12+H}}dr=C s^{1-2H},\] and
\[|X_s-X_r|\le \|X\|_{H-\epsilon,T} (s-r)^{H-\epsilon}, \quad 0\le r<s\le T,\]
we have, if we choose $\epsilon<\frac12$,
\beqys
\left|\left(K_H^{-1}\int_0^\cdot  X_rdr\right)(s)\right|
&\le& C \left(s^{\frac12-H}s^{H-\epsilon}\|X\|_{H-\epsilon,T}+ s^{H-\frac12}\|X\|_{H-\epsilon,T} \int_0^s (s-r)^{H-\epsilon}(s-r)^{-\frac12-H}r^{\frac12-H}dr\right)\\
&\le &C\|X\|_{H-\epsilon,T}(s^{\frac12-\epsilon}+ s^{H-\frac12})\\
&\le & C \|X\|_{H-\epsilon,T}.
\eeqys
Then when $\Delta t$ small enough,  (\ref{novikov}) follows from Lemma \ref{fernique}.


Now choose $n$ big enough such that $\dfrac Tn\le \Delta t$, and let $t_i=\dfrac{iT}{n}, i=1,\cdots, n.$
Denoting  $\left(K_H^{-1}\int_0^\cdot \frac\alpha\sigma X_r dr\right)(s)$ by $f(s)$, we obtain
\beqys
\E[\xi_T]&=&\E\exp\left(\int_0^Tf(s)dW_s-\frac12\int_0^Tf^2(s)ds\right)\\
&=&\E\left(\E\left[\left.\exp\left(\int_0^Tf(s)dW_s-\frac12\int_0^Tf^2(s)ds\right)\right|\mathcal F_{t_{n-1}}\right]\right)\\
&=&\E\bigg(\exp\left(\int_0^{t_{n-1}}f(s)dW_s-\frac12\int_0^{t_{n-1}}f^2(s)ds\right)\\
&&\quad \times\E\left[\left.\exp\left(\int_{t_{n-1}}^T f(s)dW_s-\frac12\int_{t_{n-1}}^Tf^2(s)ds\right)\right|\mathcal F_{t_{n-1}}\right]\bigg).
\eeqys
We have that $\displaystyle\E\left[\left.\exp\left(\int_{t_{n-1}}^T f(s)dW_s-\frac12\int_{t_{n-1}}^Tf^2(s)ds\right)\right|\mathcal F_{t_{n-1}}\right]\le 1$ a.s., since\\ $\displaystyle\exp\left(\int_{0}^\cdot f(s)dW_s-\frac12\int_{t_{n-1}}^\cdot f^2(s)ds\right)$ is a positive local martingale and hence a supermartingale.

On the other hand, since $ \displaystyle \E\left[\exp\left(\frac12\int_{t_{n-1}}^Tf^2(s)ds\right)\right]<\infty$, then by Novikov Criterion, we know that $\displaystyle\E\left[\exp\left(\int_{t_{n-1}}^T f(s)dW_s-\frac12\int_{t_{n-1}}^Tf^2(s)ds\right)\right]= 1$. This implies that \[\E\left[\left.\exp\left(\int_{t_{n-1}}^T f(s)dW_s-\frac12\int_{t_{n-1}}^Tf^2(s)ds\right)\right|\mathcal F_{t_{n-1}}\right]= 1 \quad \mbox{a.s.}\]

We use the above procedure $n$ times, and obtain that  $\E(\xi_T)=1.$
\hfill $\Box$

{\color{black} \begin{lemma}\label{ergodicity} For $H\in(0,1)$, we have
\begin{equation}\label{e7.11''}
\limsup_{T\to \infty} \frac1{T^{H+1}\ln T} \int_0^T |W_t^H|dt\le \E[|W_1^H|],\quad a.s.,
\end{equation} and
\begin{equation}\label{e7.11'}
\liminf_{T\to \infty} \frac1{T\ln T} \int_0^T |W_t^H|dt\ge\frac{H}{H+1}\E[|W_1^H|],\quad a.s.
\end{equation}
\end{lemma}
\begin{proof}
The Gaussian process $\{e^{-Ht}W^H_{e^t}, t\ge 0\}$ is stationary and ergodic. Indeed, the covariance function is given by
\[\rho(t,s):=\E\left[e^{-Ht}W^H_{e^t}e^{-Hs}W^H_{e^s}\right]=\frac12\left(e^{H(t-s)}+e^{H(s-t)}-\left[e^{t-s}+e^{s-t}-2\right]^H\right).\]
The covariance function $\rho$ is a function of $|t-s|$, and hence the Gaussian process $\{e^{-Ht}W^H_{e^t}, t\ge 0\}$ is stationary. On the other hand, the covariance function $\rho$ vanishes to $0$ as $|t-s|$ goes to $\infty$, because
\[\lim_{x\to \infty} \left(x^H+x^{-H}-[x+x^{-1}-2]^H\right)=0.\] This implies that the Gaussian process is ergodic. By the ergodic theorem, we have
\[\lim_{T\to \infty} \frac1T \int_0^T e^{-Ht}|W_{e^t}^H|dt=\E[|W_1^H|],\quad a.s.,\]
which is equivalent to
\begin{equation}\label{e7.12}
\lim_{T\to \infty} \frac1{\ln T} \int_1^T s^{-H-1}|W_{s}^H|ds=\E[|W_1^H|],\quad a.s.
\end{equation}
Therefore,
\begin{align*}
\limsup_{T\to \infty} \frac1{T^{H+1}\ln T} \int_0^{T} |W_s^H|ds&=  \limsup_{T\to \infty} \frac1{\ln T} T^{-(H+1)}\int_1^{T} |W_s^H|ds\notag\\
&\le \limsup_{T\to \infty}  \frac1{\ln T}  \int_{1}^{T} s^{-H-1} |W_s^H|ds\notag\\
&=\E[|W_1^H|] \quad a.s.,
\end{align*}
and (\ref{e7.11''}) is proved.

Now we show (\ref{e7.11'}). For $\alpha\in(0,1)$, the equality (\ref{e7.12}) implies
\[\lim_{T\to \infty} \frac1{\ln T^\alpha} \int_1^{T^\alpha} s^{-H-1}|W_{s}^H|ds=\E[|W_1^H|],\quad a.s.,\]
and hence
\begin{equation}\label{e7.13}
\lim_{T\to \infty} \frac1{\ln T} \int_1^{T^\alpha} s^{-H-1}|W_{s}^H|ds=\alpha\E[|W_1^H|],\quad a.s.
\end{equation}

By the relations (\ref{e7.12}) and (\ref{e7.13}),
\begin{equation}\label{e7.14}
\lim_{T\to \infty} \frac1{\ln T} \int_{T^\alpha}^T s^{-H-1}|W_{s}^H|ds=(1-\alpha)\E[|W_1^H|],\quad a.s.
\end{equation}
If we choose $\alpha=\frac1{H+1}$, then
\begin{align*}
\liminf_{T\to \infty} \frac1{T\ln T} \int_0^T |W_s^H|ds
&\ge \liminf_{T\to \infty}  \frac1{\ln T} T^{-\alpha(H+1)} \int_{T^\alpha}^T |W_s^H|ds\\
&\ge \liminf_{T\to \infty} \frac1{\ln T}\int_{T^\alpha}^T s^{-H-1}|W_{s}^H|ds\\
&=\frac{H}{H+1}\E[|W_1^H|] \quad a.s.,
\end{align*}
where the last equality follows from (\ref{e7.14}).
\end{proof}
}

\section*{Acknowledgment}\label{ack}
The authors would like to thank Professors David Nualart and Yaozhong Hu for their helpful comments and suggestions on an earlier draft of this paper. The research work of C. Lee is supported in part by Army Research Office (W911NF-14-1-0216). 

\bibliographystyle{plain}
\bibliography{chRefer20091}

\end{document}